\documentclass{amsart}
\usepackage[a4paper, total={6in, 8in}]{geometry}
\usepackage[utf8]{inputenc}

\usepackage{cite}

\title{Contractible complexes and non-positive immersions}
\author{William Fisher}
\address{Trinity College, Cambridge}
\email{will\_fisher@berkeley.edu}
\date{July 2023}
\subjclass[2020]{20F65}
\keywords{Non-positive immersions, Whitehead's conjecture}

\usepackage{xcolor}

\usepackage{mathtools}

\usepackage{verbatim}

\usepackage{tikz-cd}
\tikzset{
math to/.tip={Glyph[glyph math command=rightarrow]},
loop/.tip={Glyph[glyph math command=looparrowleft, swap]},
loop'/.tip={Glyph[glyph math command=looparrowleft]},
 weird/.tip={Glyph[glyph math command=Rrightarrow, glyph length=1.5ex]},
  pi/.tip={Glyph[glyph math command=pi, glyph length=1.5ex, glyph axis=0pt]},
}

\usepackage{schemata}
\makeatletter
\newcommand{\eqnum}{\refstepcounter{equation}\textup{\tagform@{\theequation}}}
\makeatother

\makeatletter
\newsavebox{\@brx}
\newcommand{\llangle}[1][]{\savebox{\@brx}{\(\m@th{#1\langle}\)}%
  \mathopen{\copy\@brx\kern-0.5\wd\@brx\usebox{\@brx}}}
\newcommand{\rrangle}[1][]{\savebox{\@brx}{\(\m@th{#1\rangle}\)}%
  \mathclose{\copy\@brx\kern-0.5\wd\@brx\usebox{\@brx}}}
\makeatother

\usepackage{amsmath, amsfonts, amssymb, amsthm, mathrsfs}

\DeclareMathOperator{\depth}{depth}

\usepackage{enumitem}

\newtheorem{theorem}{Theorem}[section]
\newtheorem{lemma}[theorem]{Lemma}

\newtheorem{proposition}[theorem]{Proposition}
\newtheorem{conjecture}[theorem]{Conjecture}

\theoremstyle{definition}
\newtheorem{definition}[theorem]{Definition}
\newtheorem{construction}[theorem]{Construction}

\theoremstyle{remark}
\newtheorem{remark}[theorem]{Remark}

\usepackage{scalerel}

\usetikzlibrary{decorations.markings}
\usetikzlibrary{shapes.geometric}
\usetikzlibrary{calc}
\usetikzlibrary{positioning}
\tikzset{->-/.style={decoration={
  markings,
  mark=at position .5 with {\arrow{>}}},postaction={decorate}}}
  
\tikzset{-*-/.style={decoration={
  markings,
  mark=at position #1 with {\arrow{>}}},postaction={decorate}}}

\begin{document}

\maketitle

\begin{abstract}
    We provide examples of contractible complexes which fail to have non-positive immersions and weak non-positive immersions, answering a conjecture of Wise in the negative.
    
    \medskip

	\noindent\textsc{R\'esum\'e.} Nous fournissons des exemples de complexes contractibles qui ne poss\`edent pas d'immersions non-positives et d'immersions faiblement non-positives, r\'epondant ainsi \`a une conjecture de Wise de mani\`ere n\'egative.
\end{abstract}

\section{Introduction}

In \cite{wise2022coherence}, Wise introduced the following property of 2-complexes.

\begin{definition}[Non-positive immersions]
    A map $Y\to X$ of CW-complexes is an \emph{immersion} if it is locally injective, written $Y\looparrowright X$.\par
    A 2-complex $X$ is said to have \emph{non-positive immersions} (NPI) if for every immersion $Y\looparrowright X$ of a compact, connected 2-complex $Y$ into $X$, we have either $\chi(Y)\le 0$ or $\pi_1 Y = 1$.
\end{definition}

One also may consider the stronger notion of \emph{contracting} NPI which says for every immersion $Y\looparrowright X$ of a compact, connected 2-complex $Y$ either $\chi(Y)\le 0$ or $Y$ is contractible. We will also say that $X$ has \emph{weak} NPI (WNPI) if $\chi(Y)\le 1$ for every compact, connected $Y\looparrowright X$. Despite the terminology, one should note that the use of ``weak'' here denotes only a weakening of the stronger contracting NPI property. By Hurewicz, one in fact has that contracting NPI is equivalent to having NPI and weak NPI.\par
The motivation for studying what classes of contractible complexes have NPI and WPNI comes from a famous conjecture of Whitehead \cite{whitehead1941adding}.

\begin{conjecture}[Whitehead]\label{conj:whiteheadconj}
    Every subcomplex of an aspherical 2-complex $X$ is aspherical.
\end{conjecture}

In addressing this conjecture, it thus makes sense to look for properties of complexes inherited by subcomplexes which imply asphericity. In particular, the property of having WNPI implies the complex $X$ is aspherical \cite[Theorem 3.4]{wise2022coherence} and is inherited by subcomplexes. Passing to universal covers, one is then led to a conjecture of Wise \cite{wise2022coherence}.

\begin{conjecture}[Wise]\label{conj:wiseconj}
    Every contractible 2-complex $X$ has (weak) non-positive immersions.
\end{conjecture}

An affirmative answer to the WNPI variant of Conjecture \ref{conj:wiseconj} would imply Conjecture \ref{conj:whiteheadconj}. Indeed, assuming all contractible complexes to have WNPI, then given an aspherical 2-complex $X$ with sub-complex $Y$, we can consider the contractible universal cover $\hat X\to X$ of $X$. $Y$ is then covered by a sub-complex $\hat Y\subseteq \hat X$. Since $\hat X$ has WNPI, it follows that so does $\hat Y\looparrowright \hat X$. Thus $\hat Y$, and hence $Y$, is aspherical.\par
Due to work of Howie, one can also deduce Conjecture \ref{conj:whiteheadconj} from the NPI variant of Conjecture \ref{conj:wiseconj}. In \cite[Question 2]{howie1979aspherical}, Howie asks whether the fundamental group of a subcomplex of a contractible 2-complex has no non-trivial finitely generated perfect subgroups and shows an affirmative answer to this question implies Conjecture \ref{conj:whiteheadconj}. However, by Wise \cite[Theorem 3.3]{wise2022coherence}, if $X$ has NPI, then $\pi_1 X$ is locally indicable. Since NPI is inherited by subcomplexes, Conjecture \ref{conj:wiseconj} in the NPI form implies a positive answer to \cite[Question 2]{howie1979aspherical}, implying Conjecture \ref{conj:whiteheadconj}.\par
The purpose of this note is to provide examples of contractible 2-complexes which fail to have WNPI and NPI, answering both forms of Conjecture \ref{conj:wiseconj} in the negative.\par
After this work was completed, the author learned that Chemtov and Wise have also constructed a family of counterexamples to Conjectures \ref{conj:wiseconj} using different techniques \cite{wisechemtov}.

\subsection*{Acknowledgements}
This research was supported by the Trinity College Summer Studentship Scheme. The author is also grateful to Henry Wilton for many helpful comments and discussions.

\section{Examples}\label{sec:examples}
The contractible complexes searched come from a family introduced by Miller--Schupp \cite{miller1999some}. They showed that balanced presentations of the form
\begin{equation}\label{eq:prestype}
    \langle a,b\, |\, w,\, ba^nb^{-1}a^{-(n + 1)}\rangle
\end{equation}
where $n\ge 1$ and $w$ is a word with exponent sum $1$ in $b$ represent the trivial group, and hence have contractible presentation complexes by theorems of Hurewicz and Whitehead. Throughout the remainder of this paper, all morphisms of complexes will be assumed to be \emph{combinatorial}, i.e.\ sending $n$-cells homeomorphically to $n$-cells.\par
To begin, we remark the following lemma.
\begin{lemma}
    Suppose $X$ is a 2-complex failing to have WNPI. Then there exists a 2-complex $X'$ homotopy equivalent to $X$ which has neither WNPI nor NPI.
\end{lemma}
\begin{proof}
    Suppose $Y\looparrowright X$ with $\chi(Y) \ge 2$. Then consider the natural immersion $Y\vee S^1\looparrowright X\vee D^2$ given by sending $Y$ via $Y\looparrowright X$ and $S^1$ to the boundary of $D^2$. Then letting $X' = X\vee D^2$ we have $X' \simeq X$ and $\chi(Y\vee S^1) \ge 1$, yet $\pi_1(Y\vee S^1)\neq 1$.
\end{proof}
\noindent Thus failure of the WNPI variant of Conjecture \ref{conj:wiseconj} implies failure of the NPI variant as well. As such, the goal of this section will be to present contractible 2-complexes failing to have WNPI.\par
The case of $n = 1$ and $w = ab^2ab^{-1}$ as in \eqref{eq:prestype} fails to have WNPI as shown by Figure \ref{fig:weak_npi_example}. Figure \ref{fig:weak_npi_example} gives a combinatorial description of the immersion by showing the commuting diagram of 1-skeleta and attaching maps of 2-cells induced by the immersion $Y\looparrowright X$. Magenta is used to denote the images of cells under the various maps, and Figure \ref{fig:linkdiagrams} shows the links of all appearing vertices. Note that $Y\looparrowright X$ since the induced maps on links are injective and in this example $\chi(Y)=2$. Section \ref{sec:algodescription} gives a brief description of how this example, and others in this section, were generated.\par
An exhaustive analysis of presentation complexes of the form \eqref{eq:prestype} with $n = 1$ and $|w| \le 6$ was also conducted. To reduce the search space, one can show that for a general 2-complex having NPI and WNPI depends only on the isomorphism class of the group pair (c.f.\ \cite{wilton2022rational}), which, up to identifying all free groups of a given rank, corresponds to the $\mathrm{Aut}(\pi_1 X^1)$-orbit of the set of conjugacy classes of $\langle w_1\rangle,\dots,\langle w_n\rangle$ where $w_1,\dots,w_n$ are the attaching words of the 2-cells of $X$. As such, only one representative $w$ of each such orbit needs to be tested. Those $w$ that failed to have WNPI are given by
\begin{equation}\label{eq:failingwords}
	w\in\{ab^2ab^{-1}, a^{-1}b^2a^{-1}b^{-1}, a^2b^{-1}ab^2, ab^{-1}a^{-2}b^2, a^2b^2a^{-1}b^{-1}\}.
\end{equation}
Table \ref{tab:fail_npi_examples} summarizes these findings and groups all $|w| \le 6$ into their respective group pair isomorphism classes.

We remark that the case of $w = b$, which accounts for all presentations of the form \eqref{eq:prestype} with $n = 1$ and $|w|\le 4$ up to group pair isomorphism, is still open. Testing however shows that there exist immersions $Y\looparrowright X$ with $\chi(Y) = 1$ and $Y$ having no free faces, making the example resistant to current approaches for proving NPI and WNPI.

\begin{table}
	\centering
    \begin{tabular}{|c||c|c|}\hline
    	Representative $w$ & $w$ with $|w|\le 6$ and Isomorphic Group Pair & Known to Fail WNPI \\\hline
        $b$ & $b$, $ab$, $a^-1b$, $a^2b$, $a^{-2}b$, $a^3b$, $a^{-3}b$, $a^4b$, $a^{-4}b$, $a^5b$, $a^{-5}b$ & No \\
        $ab^2ab^{-1}$ & $ab^2ab^{-1}$, $(ab)^2ab^{-1}$, $aba^{-1}bab^{-1}$ & Yes \\
        $ab^2a^{-1}b^{-1}$ & $ab^2a^{-1}b^{-1}$, $(ab)^2a^{-1}b^{-1}$, $a(ba^{-1})^2b^{-1}$ & No \\
        $ab^{-1}a^{-1}b^2$ & $ab^{-1}a^{-1}b^2$, $abab^{-1}a^{-1}b$ & No \\
        $a^{-1}b^2a^{-1}b^{-1}$ & $a^{-1}b^2a^{-1}b^{-1}$, $aba^{-1}b^{-1}a^{-1}b$, $(a^{-1}b)^2a^{-1}b^{-1}$ & Yes \\
        $a^2b^2ab^{-1}$ & $a^2b^2ab^{-1}$ & No \\
        $a^2b^2a^{-1}b^{-1}$ & $a^2b^2a^{-1}b^{-1}$ & Yes \\
        $a^2b^{-1}ab^2$ & $a^2b^{-1}ab^2$ & Yes \\
        $a^2b^{-1}a^{-1}b^2$ & $a^2b^{-1}a^{-1}b^2$ & No \\
        $ab^2a^{-2}b^{-1}$ & $ab^2a^{-2}b^{-1}$ & No \\
        $ab^{-1}a^{-2}b^2$ & $ab^{-1}a^{-2}b^2$ & Yes \\
        $ab^{-1}(a^{-1}b)^2$ & $ab^{-1}(a^{-1}b)^2$ & No \\
        $a^{-2}b^2a^{-1}b^{-1}$ & $a^{-2}b^2a^{-1}b^{-1}$ & No \\
        $a^{-2}b^{-1}a^{-1}b^2$ & $a^{-2}b^{-1}a^{-1}b^2$ & No \\\hline
    \end{tabular}
    \caption{Group pair isomorphism classes of presentations of the form \eqref{eq:prestype} with $n = 1$ and $|w|\le 6$ and whether they are known to fail to have WNPI.}
    \label{tab:fail_npi_examples}
\end{table}

\begin{figure}
    \centering
    \scalebox{0.9}{
        \begin{tikzpicture}[node distance={25mm}, thick, main/.style = {draw, circle}] 
\node (Y_skeleton) at (0,0) {
    \begin{tikzpicture}[->,node distance={25mm and 25mm}, thick, main/.style = {draw, circle}] 
\node[main] (1) {$v_1$}; 
\node[main] (2) [right of=1] {$v_2$}; 
\node[main] (3) [right of=2] {$v_3$}; 
\node[main] (4) [below of=2] {$v_4$}; 
\path[every node/.style={font=\sffamily}]
    (2) edge[bend right=20] node [pos=0.2, above] {$e_4$} (1)
    (1) edge[bend right=20] node [pos=0.3, below] {$e_8$} (2)
    (2) edge[bend right=20] node [pos=0.3, below] {$e_7$} (3)
    (2) edge[bend left=20] node [pos=0.3, above] {$e_3$} (3)
    (1) edge node [pos=0.3, below left] {$e_1$} (4)
    (3) edge node [pos=0.3, below right] {$e_2$} (4)
    (3) edge[bend right=50] node [pos=0.3, above] {$e_5$} (1);

\draw (4) to [out=310,in=230,looseness=6] node[pos=0.3, below right=.1mm] {$e_6$} (4);

\path[-*-=0.75, magenta, bend right=20] (2) to node[pos=0.55, above=.5mm] {$b^{-1}$} (1);

\path[-*-=0.75, magenta, bend right=20] (1) to node[pos=0.65, below=.5mm] {$a$} (2);

\path[-*-=0.75, magenta, bend right=20] (2) to node[pos=0.65, below=.5mm] {$a$} (3);

\path[-*-=0.75, magenta, bend left=20] (2) to node[pos=0.65, above=.5mm] {$b$} (3);

\path[-*-=0.75, magenta] (1) to node[pos=0.7, below left=.25mm] {$a^{-1}$} (4);

\path[-*-=0.75, magenta] (3) to node[pos=0.7, below right=.25mm] {$a$} (4);

\path[-*-=0.75, magenta, bend right=50] (3) to node[pos=0.7, above=.25mm] {$b$} (1);

\path[-*-=0.75, magenta] (4) to [out=310,in=230,looseness=6] node[pos=0.7, below left=.1mm] {$b^{-1}$} (4);

\end{tikzpicture}
};
\node (X_skeleton) at (8,0) {
    \input{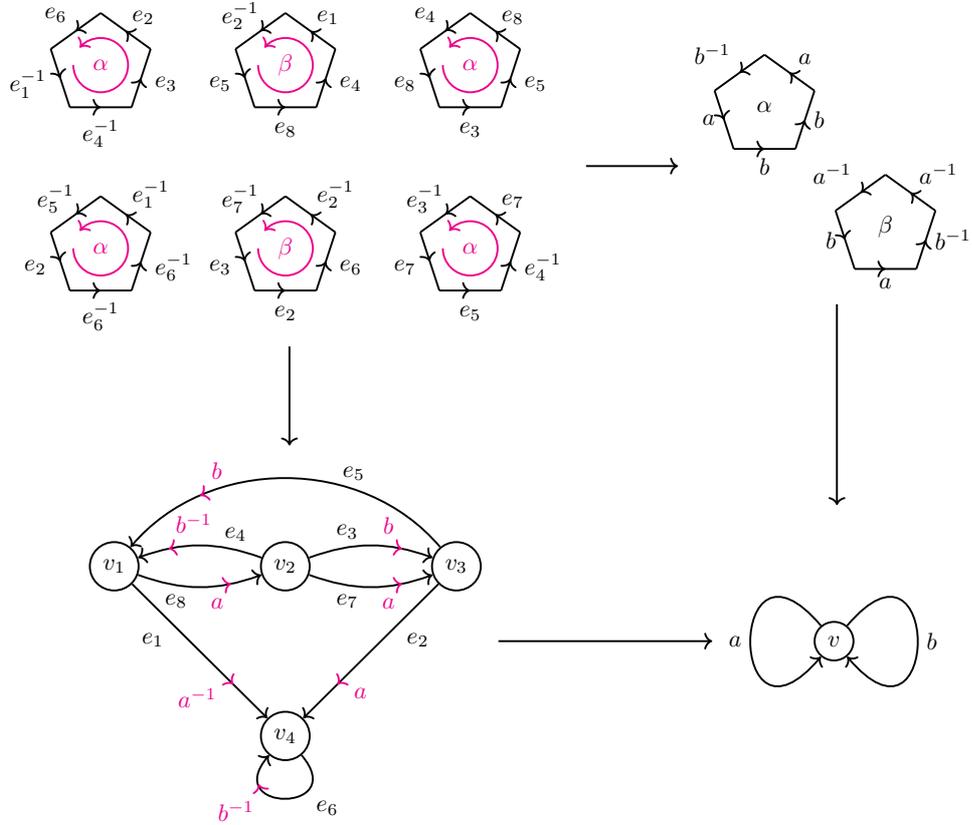}
};
\node (Y_faces) at (0, 7) {
    \begin{tikzpicture}[node distance={27mm and 23mm}, thick]
    \node (f1) [
        magenta,
        draw=none,
        minimum size=1.5cm,
        regular polygon, regular polygon sides=5,
    ]{$\alpha$};

    \draw[->-] (f1.corner 1) -- (f1.corner 2) node[midway, above left=-.0] {$e_6$};
    \draw[->-] (f1.corner 2) -- (f1.corner 3) node[pos=.6, left=.1] {$e_1^{-1}$};
    \draw[->-] (f1.corner 3) -- (f1.corner 4) node[midway, below=.05] {$e_4^{-1}$};
    \draw[->-] (f1.corner 4) -- (f1.corner 5) node[pos=0.4, right=.1] {$e_3$};
    \draw[->-] (f1.corner 5) -- (f1.corner 1) node[midway, above right=-.05] {$e_2$};
    
    \node (f2) [
        draw=none,
        minimum size=1.5cm,
        regular polygon, regular polygon sides=5,
        right of=f1,
        magenta,
    ]{$\beta$};
    
    \draw[->-] (f2.corner 1) -- (f2.corner 2) node[midway, above left=-.15] {$e_2^{-1}$};
    \draw[->-] (f2.corner 2) -- (f2.corner 3) node[pos=.6, left=.1] {$e_5$};
    \draw[->-] (f2.corner 3) -- (f2.corner 4) node[midway, below=.1] {$e_8$};
    \draw[->-] (f2.corner 4) -- (f2.corner 5) node[pos=0.4, right=.1] {$e_4$};
    \draw[->-] (f2.corner 5) -- (f2.corner 1) node[midway, above right=-.05] {$e_1$};
    
    \node (f3) [
        draw=none,
        minimum size=1.5cm,
        regular polygon, regular polygon sides=5,
        right of=f2,
        magenta
    ]{$\alpha$};
    
    \draw[->-] (f3.corner 1) -- (f3.corner 2) node[midway, above left=0] {$e_4$};
    \draw[->-] (f3.corner 2) -- (f3.corner 3) node[pos=.6, left=.1] {$e_8$};
    \draw[->-] (f3.corner 3) -- (f3.corner 4) node[midway, below=.1] {$e_3$};
    \draw[->-] (f3.corner 4) -- (f3.corner 5) node[pos=0.4, right=.1] {$e_5$};
    \draw[->-] (f3.corner 5) -- (f3.corner 1) node[midway, above right=-.05] {$e_8$};
    
    \node (f4) [
        draw=none,
        minimum size=1.5cm,
        regular polygon, regular polygon sides=5,
        below of=f1,
        magenta
    ] {$\alpha$};
    
    \draw[->-] (f4.corner 1) -- (f4.corner 2) node[midway, above left=-.15] {$e_5^{-1}$};
    \draw[->-] (f4.corner 2) -- (f4.corner 3) node[pos=.6, left=.1] {$e_2$};
    \draw[->-] (f4.corner 3) -- (f4.corner 4) node[midway, below=.05] {$e_6^{-1}$};
    \draw[->-] (f4.corner 4) -- (f4.corner 5) node[pos=0.4, right=.1] {$e_6^{-1}$};
    \draw[->-] (f4.corner 5) -- (f4.corner 1) node[midway, above right=-.05] {$e_1^{-1}$};
    
    \node (f5) [
        draw=none,
        minimum size=1.5cm,
        regular polygon, regular polygon sides=5,
        right of=f4,
        magenta
    ] {$\beta$};
    
    \draw[->-] (f5.corner 1) -- (f5.corner 2) node[midway, above left=-.15] {$e_7^{-1}$};
    \draw[->-] (f5.corner 2) -- (f5.corner 3) node[pos=.6, left=.1] {$e_3$};
    \draw[->-] (f5.corner 3) -- (f5.corner 4) node[midway, below=.1] {$e_2$};
    \draw[->-] (f5.corner 4) -- (f5.corner 5) node[pos=0.4, right=.1] {$e_6$};
    \draw[->-] (f5.corner 5) -- (f5.corner 1) node[midway, above right=-.05] {$e_2^{-1}$};

    
    \node (f6) [
        draw=none,
        minimum size=1.5cm,
        regular polygon, regular polygon sides=5,
        right of=f5,
        magenta
    ] {$\alpha$};
    
    \draw[->-] (f6.corner 1) -- (f6.corner 2) node[midway, above left=-.15] {$e_3^{-1}$};
    \draw[->-] (f6.corner 2) -- (f6.corner 3) node[pos=.6, left=.1] {$e_7$};
    \draw[->-] (f6.corner 3) -- (f6.corner 4) node[midway, below=.1] {$e_5$};
    \draw[->-] (f6.corner 4) -- (f6.corner 5) node[pos=0.4, right=.1] {$e_4^{-1}$};
    \draw[->-] (f6.corner 5) -- (f6.corner 1) node[midway, above right=-.05] {$e_7$};
    
    \draw[magenta, ->] ([shift={(180:4mm)}]f1) arc (180:500:4mm);
    \draw[magenta, ->] ([shift={(180:4mm)}]f2) arc (180:500:4mm);
    \draw[magenta, ->] ([shift={(180:4mm)}]f3) arc (180:500:4mm);
    \draw[magenta, ->] ([shift={(180:4mm)}]f4) arc (180:500:4mm);
    \draw[magenta, ->] ([shift={(180:4mm)}]f5) arc (180:500:4mm);
    \draw[magenta, ->] ([shift={(180:4mm)}]f6) arc (180:500:4mm);
\end{tikzpicture}
};

\node (X_faces) at (8,7) {
    \begin{tikzpicture}[node distance=25mm and 25mm, thick]
    \node (f1) [
        draw=none,
        minimum size=1.5cm,
        regular polygon, regular polygon sides=5,
    ]{$\alpha$};

    \draw[->-] (f1.corner 1) -- (f1.corner 2) node[midway, above left] {$b^{-1}$};
    \draw[->-] (f1.corner 2) -- (f1.corner 3) node[midway, left] {$a$};
    \draw[->-] (f1.corner 3) -- (f1.corner 4) node[midway, below] {$b$};
    \draw[->-] (f1.corner 4) -- (f1.corner 5) node[midway, right] {$b$};
    \draw[->-] (f1.corner 5) -- (f1.corner 1) node[midway, above right] {$a$};
    
    \node (f2) [
        draw=none,
        minimum size=1.5cm,
        regular polygon, regular polygon sides=5,
        below right of=f1,
    ]{$\beta$};
    
    \draw[->-] (f2.corner 1) -- (f2.corner 2) node[midway, above left] {$a^{-1}$};
    \draw[->-] (f2.corner 2) -- (f2.corner 3) node[midway, left] {$b$};
    \draw[->-] (f2.corner 3) -- (f2.corner 4) node[midway, below] {$a$};
    \draw[->-] (f2.corner 4) -- (f2.corner 5) node[midway, right] {$b^{-1}$};
    \draw[->-] (f2.corner 5) -- (f2.corner 1) node[midway, above right] {$a^{-1}$};
\end{tikzpicture}
};

\draw[->] (Y_skeleton) to (X_skeleton);
\draw[->] (Y_faces) to (Y_skeleton);
\draw[->] (X_faces) to (X_skeleton);
\draw[->] (Y_faces) to (X_faces);

\end{tikzpicture}
    }
    \caption{Failure of WNPI for the presentation complex of $\langle a,b\,|\,ab^2ab^{-1},\, bab^{-1}a^{-2}\rangle$.}
    \label{fig:weak_npi_example}
\end{figure}

\begin{figure}
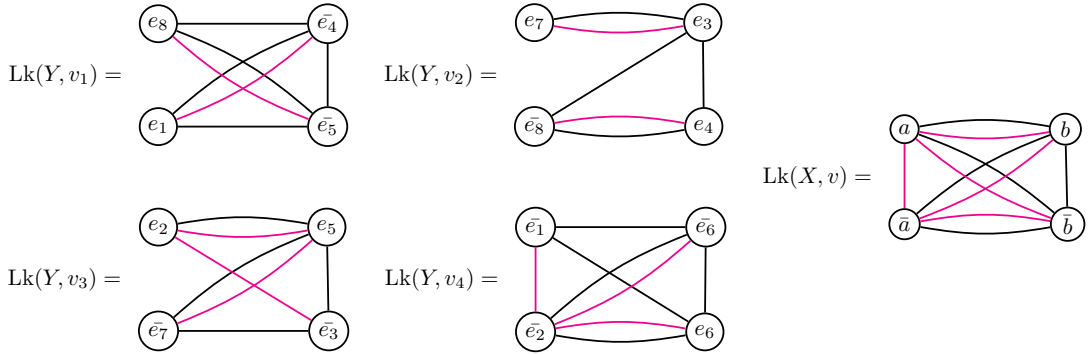

    \centering
    \scalebox{0.85}{
        \begin{tikzpicture}[node distance={7mm and 22mm}, thick] 
\node[label=left:{$\mathrm{Lk}(Y,v_1) = $}] (v1) {
    \input{links/linkv1}
};

\node[label=left:{$\mathrm{Lk}(Y,v_2) = $}] (v2) [right=of v1] {
    \input{links/linkv2}
};

\node[label=left:{$\mathrm{Lk}(Y,v_3) = $}] (v3) [below=of v1] {
    \input{links/linkv3}
};

\node[label=left:{$\mathrm{Lk}(Y,v_4) = $}] (v4) [right=of v3] {
    \input{links/linkv4}
};

\node[label=left:{$\mathrm{Lk}(X,v) = $}] (v) [xshift=115mm] at ($(v1)!0.5!(v3)$) {
    \input{links/linkv}
};

\end{tikzpicture}
    }
    \caption{Links of the vertices of complexes in the immersion $Y\looparrowright X$ given by Figure \ref{fig:weak_npi_example}.}
    \label{fig:linkdiagrams}
\end{figure}

\section{Generating immersions}\label{sec:algodescription}
The goal of this section is to give a brief summary of how the examples of Section \ref{sec:examples} were generated.\par
To start, an important tool for algorithmically generating immersions into a complex is the notion of folding. For graphs, this notion was popularized by Stallings' \cite{stallings1983topology} and takes a morphism $f: G\to H$ and factors it as
\begin{equation*}
    G \to \bar G\looparrowright H
\end{equation*}
where $G\to\bar G$ is surjective and $\pi_1$-surjective, and $\bar G\looparrowright H$ is an immersion. Moreover, this factorization is the most ``general'' possible factorization through an immersion.\par
This can be generalized to 2-complexes. Recall that a morphism of complexes in this setting is assumed to be combinatorial. Given this, the appropriate generalization of folding of graphs is given by \cite[Lemma 4.1]{louder2018one}.

\begin{lemma}\label{lemma:folding2complexes}
    Let $\phi : A\to B$ be a morphism of 2-complexes with $A$ finite. Then $\phi$ factors as
    \begin{equation*}
        A\to  \bar A\looparrowright B
    \end{equation*}
    where $A\to \bar A$ is surjective and $\pi_1$-surjective. Furthermore, if $A\to B$ factors through an immersion $\psi : Z\looparrowright X$, then $\bar A\looparrowright B$ factors uniquely through $\psi$.
\end{lemma}

In the notation above, $\bar A\looparrowright B$ will be referred to as the folding of $\phi$. Moreover, this factorization can be efficiently algorithmically computed. This allows one to algorithmically produce immersions into a complex $X$ from combinatorial maps $Y\to X$. The next ingredient then is to produce maps $Y\to X$ whose folding produces complexes with large Euler characteristic. Given that we can always remove \emph{isolated edges}, i.e.\ edges not adjacent to any 2-cell of $Y$, to increase the Euler characteristic, we propose the following construction.

\begin{construction}\label{constr:mainconstruction} We construct a forest $\mathcal{G}(X)$ of immersions $Y\looparrowright X$ with $Y$ finite, connected and having no isolated edges inductively as follows:
\begin{enumerate}[label=(\roman*)]
    \item For every 2-cell $F$ of $X$, consider the map $D\to X$ sending a disc $D$ onto $F\subseteq X$. The root nodes of $\mathcal{G}(X)$ are given by the foldings of all such maps.
    \item Given a node $Y\looparrowright X \in \mathcal{G}(X)$, its children are constructed as follows: For each 2-cell $F$ of $X$ and orientation, consider the map
    \begin{equation}\label{eq:addingfacemap}
        Y\amalg D\to X.
    \end{equation}
    where $D$ is a disc mapped onto $F$. Now, consider all possible factorizations of \eqref{eq:addingfacemap} through a wedge
    \begin{equation*}
        Y\vee D \to X
    \end{equation*}
    at some vertex of $Y$ and $D$. For each such map, consider the folding $C\looparrowright X$. If $C$ still has one more 2-cell than $Y$, then $C\looparrowright X$ becomes a child of $Y\looparrowright X$.
\end{enumerate}
\end{construction}

\begin{remark}
    In practice, in Step (ii) of Construction \ref{constr:mainconstruction} one should check all children for high Euler characteristic before discarding those which lost 2-cells in the folding process. While examples may still be found inside $\mathcal{G}(X)$ itself, examples which have additional identified 2-cells tend to be smaller. Figure \ref{fig:weak_npi_example} was found in this manner and Figure \ref{fig:step-by-step} shows the step-by-step construction per Construction \ref{constr:mainconstruction} leading up to it.
\end{remark}

Traversing this forest gives a sequence of immersions $Y_i\looparrowright X$ with $\chi(Y_i)$ non-decreasing, as the next proposition shows. For $Y\looparrowright X\in \mathcal{G}(X)$ we define $\depth(Y\looparrowright X)$ to be the length of the shortest path between $Y\looparrowright X$ and a root node of $\mathcal{G}(X)$.

\begin{proposition}
    For every $Y\looparrowright X\in \mathcal{G}(X)$, we have that
    \begin{equation}\label{eq:depthequalsfaces}
        \depth(Y\looparrowright X) = \#\{\textnormal{2-cells of }Y\} - 1
    \end{equation}
    and if $Y\looparrowright X\in\mathcal{G}(X)$ is a descendent of $Z\looparrowright X$ with $\depth(Z\looparrowright X) \le \depth(Y\looparrowright X)$, then $\chi(Z) \le \chi(Y)$.
\end{proposition}

\begin{proof}
    \eqref{eq:depthequalsfaces} is clear from construction of $\mathcal{G}(X)$ by induction since only maps where no 2-cells are folded together are kept as children.\par
    The second claim also follows by induction. Indeed, let $Y\looparrowright X$ be a child of $Z\looparrowright X$. Since only maps where no 2-cells are folded together are kept as children, the second Betti number does not decrease and as folding is $\pi_1$-surjective, the first Betti number does not increase. Thus we have that
    \begin{equation*}
    	\begin{aligned}
    		\chi(Y) &= 1 - b_1(Y) + b_2(Y) \\
    		&\ge 1 - b_1(Z) + b_2(Z) = \chi(Z).
    	\end{aligned}
    	\qedhere
    \end{equation*}
\end{proof}

Folding as in Lemma \ref{lemma:folding2complexes} and traversal of $\mathcal{G}(X)$ as above were implemented\footnote{https://github.com/willfisher/npi-and-folding} in Python and used to find the examples of contractible complexes that fail variants of NPI. Serialized versions of all examples referenced in this paper can be found at this repository.

\begin{figure}
    \centering
    \input{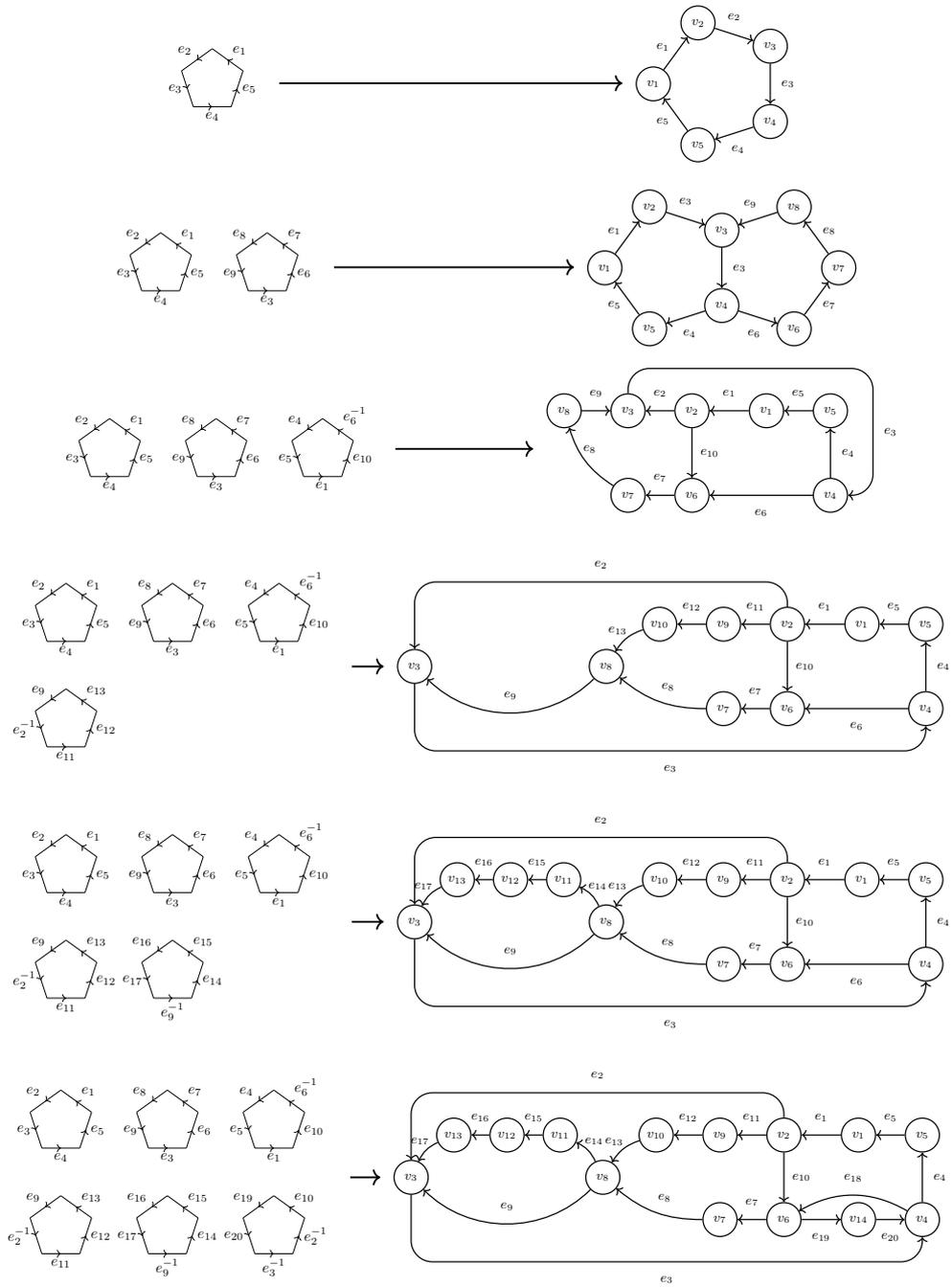}
    \caption{Step-by-step construction of the immersion in Figure \ref{fig:weak_npi_example} per algorithm described in Construction \ref{constr:mainconstruction}.}
    \label{fig:step-by-step}
\end{figure}

\bibliographystyle{plain}
\bibliography{references.bib}

\end{document}